\documentclass[twoside,12pt,leqno]{amsproc}
\usepackage{amssymb,latexsym,enumerate,verbatim,stmaryrd,tikz,mathtools}
\usetikzlibrary{patterns}
\usepackage[pagebackref]{hyperref}
\usepackage{amsrefs,wasysym}    
\usepackage{etoolbox}   
\usepackage{booktabs}
\usepackage[euler-digits,euler-hat-accent]{eulervm}
\hypersetup{citecolor=red, linkcolor=blue, colorlinks=true}
\numberwithin{table}{section}

\theoremstyle{plain}
\newtheorem{theorem}{Theorem}[section]
\newtheorem{lemma}[theorem]{Lemma}

\theoremstyle{definition} 

\newtheorem{remark}[theorem]{Remark}

\renewcommand{\ge}{\geqslant}
\renewcommand{\le}{\leqslant}

\newcommand{\PSL}{\textup{PSL}}
\newcommand{\PSU}{\textup{PSU}}

\newcommand{\Aa}{\textup{A}}
\newcommand{\Bb}{\textup{B}}
\newcommand{\Cc}{\textup{C}}
\newcommand{\Dd}{\textup{D}}
\newcommand{\Ee}{\textup{E}}
\newcommand{\Ff}{\textup{F}}
\newcommand{\Gg}{\textup{G}}

\makeatletter        
\def\@adminfootnotes{%
\let\@makefnmark\relax  \let\@thefnmark\relax
\ifx\@empty\@date\else \@footnotetext{\@setdate}\fi
\ifx\@empty\@subjclass\else \@footnotetext{\@setsubjclass}\fi
\ifx\@empty\@keywords\else \@footnotetext{\@setkeywords}\fi
\ifx\@empty\thankses\else \@footnotetext{%
\def\par{\let\par\@par}\@setthanks}%
\fi}\makeatother   

\begin{document}

\newcommand{\ord}{\mathrm{ord}}
\newcommand{\lcm}{\mathrm{lcm}}

\hyphenation{}

\title[Sylow subgroups of finite simple groups]{On the second-largest Sylow subgroup\\ of a finite simple group of Lie type}
\author{S.~P. Glasby, Alice C. Niemeyer, and Tomasz Popiel}

\address[S.~P. Glasby]{Centre for Mathematics of Symmetry and Computation, 
The University of Western Australia, 
35 Stirling Highway, Perth WA 6009, Australia. 
Also affiliated with The Department of Mathematics, University of Canberra, ACT 2601, Australia.
Email: {\tt Stephen.Glasby@uwa.edu.au; WWW: \href{http://www.maths.uwa.edu.au/~glasby/}{http://www.maths.uwa.edu.au/$\sim$glasby/} } }
\address[A.~C. Niemeyer]{Lehrstuhl B f\"ur Mathematik, 
Lehr- und Forschungsgebiet Algebra RWTH Aachen University, 
Pontdriesch 10-16, 52062 Aachen, Germany. \newline
Email: {\tt alice.niemeyer@mathb.rwth-aachen.de\newline
WWW: \href{https://wwwb.math.rwth-aachen.de/Mitarbeiter/niemeyer.php/}
{https://wwwb.math.rwth-aachen.de/Mitarbeiter/niemeyer.php/}}}
\address[T. Popiel]{School of Mathematical Sciences, 
Queen Mary University of London, 
Mile End Road, London E1 4NS, United Kingdom. 
Also affiliated with the Centre for Mathematics of Symmetry and Computation, 
The University of Western Australia, 
35 Stirling Highway, Crawley WA 6009, Australia. 
Email: {\tt tomasz.popiel@uwa.edu.au}.}

\date{\today}

\begin{abstract}
Let $T$ be a finite simple group of Lie type in characteristic $p$, and let $S$ be a Sylow subgroup of $T$ with maximal order. 
It is well known that $S$ is a Sylow $p$-subgroup except in an explicit list of exceptions, and that $S$ is always `large' in the sense that $|T|^{1/3} < |S| \le |T|^{1/2}$.
One might anticipate that, moreover, the Sylow $r$-subgroups of $T$ with $r \neq p$ are usually significantly smaller than $S$. 
We verify this hypothesis by proving that for every $T$ and every prime divisor $r$ of $|T|$ with $r \neq p$, the order of the Sylow $r$-subgroup of $T$ at most $|T|^{2\lfloor\log_r(4(\ell+1) r)\rfloor/\ell}=|T|^{{\rm O}(\log_r(\ell)/\ell)}$, where $\ell$ is the Lie rank of $T$. 
\end{abstract}

\maketitle 
\begin{center}{{\sc\tiny MSC 2010 Classification: 20D08, 20E32, 20E07}}
\end{center}

\section{Introduction} \label{S1}

Given a finite simple group $T$ of Lie type, it is natural to ask: what is the order of the largest Sylow subgroup of $T$? 
This question dates back at least to 1955 and the articles \cites{Artin1,Artin2} of Artin, who showed that if $T$ is a classical group in characteristic $p$
and $S$ is a Sylow subgroup of $T$ with maximal order, 
then $S$ is a Sylow $p$-subgroup of $T$ unless
\renewcommand{\labelitemi}{$\circ$}
\begin{itemize}
  \item $T \cong \PSL(2,p)$ with $p$ a Mersenne prime, 
  \item $T \cong \PSL(2,r-1)$ with $r$ a Fermat prime, 
  \item $T \cong \PSL(2,8)$, $T\cong \PSU(3,3)$ or $T \cong \PSU(4,2)$. 
\end{itemize}
In these cases, $S$ is a Sylow $s$-subgroup with $s=2$, $r$, $3$, $2$ or $3$, respectively. 
Artin's investigations were extended by Kimmerle {\em et~al.} \cite{KLST} in 1990 to the cases where $T$ has exceptional Lie type, with the conclusion that $S$ is always a Sylow $p$-subgroup.
Moreover, as one immediately observes upon inspecting the order formulae for the finite simple groups of Lie type \cite{Atlas}*{Table~6}, $S$ is always `large' (in both the classical and the exceptional cases), in the sense that $|S| > |T|^c$ for some constant $c$. 
Indeed one can take $c=1/3$ by~\cite{KLST}*{Theorem~3.5}. 
(On the other hand, $|S|\le|T|^{1/2}$ by~\cite{KLST}*{Theorem~3.6}.)

With the aforementioned question settled, it is natural to ask: how large are the {\em other} Sylow subgroups of $T$? 
Buekenhout~\cite{Buekenhout} approached this question by investigating the ratios $a_i = \log(p_1^{n_1})/\log(p_i^{n_i})$ in the prime factorisation $|T| = \prod_{i=1}^s p_i^{n_i}$, where the labelling is such that $p_1^{n_1} > p_2^{n_2} > \dots > p_s^{n_s}$. 
Specifically, he asked when $a_i \le \log(3)/\log(2)$, calling the corresponding primes $p_i$ {\em good contributors} to the order of $T$, and explaining that the choice of constant $\log(3)/\log(2)$ was in some sense arbitrary but motivated by computational evidence and applications in geometry. 
The conclusion was that if $T$ does not have Lie type $\Aa_1$, $\Aa_2$ or ${}^2\Aa_2$, then $|T|$ has at most one good contributor. 
More precisely, only $21$ such groups have a good contributor $r$ distinct from the characteristic; in all of these cases, $r \le 5$ by~\cite{Buekenhout}*{Theorem~4.1}. 
The cases where $T$ has type $\Aa_1$, $\Aa_2$ or ${}^2\Aa_2$ produce many further examples, and were left open. 

In addition to having geometric applications as alluded to in \cite{Buekenhout}, Buekenhout's result has also recently been used to study certain profinite groups \cite{CL}.


The purpose of this note is to prove the following result.

\begin{theorem} \label{mainThm}
  Let $q$ be a prime power, and let $T=T(q)$ be a finite simple group of
  Lie type, as listed in \textup{Table~\ref{mainTable}}. 
Let $r$ be a prime dividing $|T|$ but not $q$, and let $R$ be a Sylow $r$-subgroup of $T$. 
Then, for $K$ and $M$ as in \textup{Table~\ref{mainTable}}, we have
$|R| \le |T|^{(\lfloor\log_r(M)\rfloor+1)/K}$,
except in the cases listed in the final column of the table. 
\end{theorem}

\begin{table}
  \caption{Data for the bound $|R| \le |T|^{(\lfloor\log_r(M)\rfloor+1)/K}$ in Theorem~\ref{mainThm}, and the Lie rank $\ell$ of each group $T$ as
    per \cite{KL}*{Tables~5.1.A and 5.1.B}.}
\label{mainTable}
\begin{tabular}{llllll}
\toprule
$T$ & $K$ & $M$ & $\ell$&Conditions on $T$ & Exceptions \\
\midrule
$\Aa_n(q)$ & $n$ & $n+1$ &$n$& $n \ge 1$, $(n,q) \not \in \{(1,2),(1,3)\}$ & \\
${}^2\Aa_n(q)$ & $n/2$ & $\hspace{-5pt} \begin{array}{l} 2(n+1),\;  2\mid n \\ 2n, \hskip11.4mm 2\nmid n \end{array}$ &$\lfloor\frac{n+1}{2}\rfloor$& $n \ge 2$, $(n,q) \neq (2,2)$ & \\
$\Bb_n(q)$ & $n$ & $2n$ &$n$& $n \ge 2$, $(n,q) \neq (2,2)$ & \\
$\Cc_n(q)$ & $n$ & $2n$ &$n$& $n \ge 3$, \textup{ $q$ odd} & \\
$\Dd_n(q)$ & $n/2$ & $2(n-1)$ &$n$& $n \ge 4$ & \\
${}^2\Dd_n(q)$ & $n/2$ & $2n$ &$n-1$& $n \ge 4$ & \\
\midrule
${}^2\Bb_2(q)$ & $2$ & $4$ &$1$& $q=2^{2m+1}$, $m \ge 1$ & \\
${}^3\Dd_4(q)$ & $6$ & $12$ &$2$& & $(q,r)=(3,13)$ \\
$\Ee_6(q)$ & $12$ & $12$ &$6$& & $(q,r)=(3,13)$ \\
${}^2\Ee_6(q)$ & $12$ & $18$ &$4$& & \\
$\Ee_7(q)$ & $18$ & $18$ &$7$& & \\
$\Ee_8(q)$ & $29$ & $30$ &$8$& & $(q,r)=(2,31)$ \\
$\Ff_4(q)$ & $12$ & $12$ &$4$& & $(q,r)=(3,13)$ \\
${}^2\Ff_4(q)'$ & $6$ & $12$ &$2$& $q=2^{2m+1}$, $m \ge 0$ & \\
$\Gg_2(q)$ & $6$ & $6$ &$2$& $q \ge 3$ & $(q,r)=(3,13)$ \\
${}^2\Gg_2(q)$ & $7/2$ & $6$ &$1$& $q=3^{2m+1}$, $m \ge 1$ & \\
\bottomrule
\end{tabular}
\end{table}

\begin{remark} \label{isomorphisms}
(i) The conditions listed in the fourth column of Table~\ref{mainTable} mitigate occurrences of isomorphisms between groups in different rows. 
In particular, note that $\Gg_2(2) \cong {}^2\Aa_2(3).2$, ${}^2\Gg_2(3) \cong \Aa_1(8).3$, and that ${}^2\Bb_2(2)$ is solvable.
Note also that we include the Tits group ${}^2\Ff_4(2)'$ in Theorem~\ref{mainThm}.
%

(ii) It follows from Table~\ref{mainTable} that for all $T$ we have $K \ge \ell/2$ and $M\le 4(\ell+1)$, where $\ell$ is the Lie rank of $T$ as
in \cite{KL}*{Tables~5.1.A and 5.1.B}. 
The upper bound on $|R|$ given in Theorem~\ref{mainThm} therefore implies the bound $|R| \le |T|^{2\lfloor\log_r(4(\ell+1) r)\rfloor/\ell}$ claimed in the abstract.

(iii) Theorem~\ref{mainThm} may be viewed as a refinement of \cite{Buekenhout}*{Theorem~4.1}, in the following sense. 
Let $S$ and $R$ denote the largest and second-largest Sylow subgroups of $T$, so that $|S|=p_1^{n_1}$ and $|R|=p_2^{n_2}$ in the above notation. 
As noted earlier, $|S| > |T|^c$ for some constant $c$. 
On the other hand, as noted above, Theorem~\ref{mainThm} implies that $|R| \le |T|^{c'\log_{p_2}(\ell)/\ell}$ for some constant $c'$, where $\ell$ is the Lie rank of $T$. 
We therefore obtain the following lower bound on $a_2$, which explains why Buekenhout's {\em good contributors} are so rare:
\[
a_2 = \frac{\log|S|}{\log|R|} > \frac{\ell c\log(|T|)}{c'\log_{p_2}(\ell)\log(|T|)} = \frac{c\ell}{c'\log_{p_2}(\ell)}.
\]

(iv) The questions considered here have also been investigated for the remaining nonabelian finite simple groups, namely the alternating groups and the $26$ sporadic simple groups. 
Precise answers can, of course, be obtained for the sporadic groups, and are recorded in \cite{KLST}*{Table~L.5} and \cite{Buekenhout}*{Section~2}. 
Orders of Sylow subgroups of the alternating group $\textup{Alt}_n$ may be computed using the classical formula of Legendre \cite{Legendre} for the prime factorisation of $n!$. 
As one might anticipate, $p_1=2$ and $p_2=3$ almost always, indeed unless $n \in \{5,6,7,9\}$ \cite{Buekenhout}*{Theorem~3.7}. 
Moreover, $p_1^{n_1}\le (n!/2)^{0.363}$ by~\cite{KLST}*{Table~L.4}.
\end{remark}

We now prove some preliminary lemmas in Section~\ref{S2}, before giving the proof of Theorem~\ref{mainThm} in Section~\ref{S3}.

\section{Supporting lemmas} \label{S2}

As in \cites{Artin1,Artin2,KLST}, we consider the cyclotomic factorisations
for the finite simple groups of Lie type, \emph{cf.}~\cite{KLST}*{Definition~4.4}. 
Writing $\Phi_i$ for the $i$th cyclotomic polynomial and $d$ for the number of diagonal outer automorphisms of $T=T(q)$, this factorisation has the form
\[
d|T|=q^{e_0}\prod_{i=1}^M\Phi_i(q)^{e_i},
\]
where $e_i \ge 0$ for $i \le M$, and $e_M>0$. We set $e_i = 0$ for $i>M$.
The values of $M$ and $e_0,e_1,\ldots,e_M$ can be deduced from the usual formulae for $|T|$ by noting that 
\[
q^i-1=\prod_{k\mid i}\Phi_k(q) \quad \text{and} \quad q^i+1=\prod_{k\mid 2i, k\nmid i}\Phi_k(q).
\]
They may be also obtained from~\cite{KLST}*{Definition 4.4 and Tables~C.1 and~C.2}.
The values of $M$ are listed in Table~\ref{mainTable}. 
Table~\ref{T:fact} lists the cyclotomic factorisations for the classical
groups, and duplicates the values of $M$ for ease of reference.


\begin{table}[!t]
\caption{Numbers $d$ of diagonal outer automorphisms, and cyclotomic factorisations $d|T|=q^{e_0}\prod_{i=1}^M\Phi_i(q)^{e_i}$ for finite simple classical groups $T$. The $q$ in $\Phi_i(q)$ are suppressed for brevity.}
\label{T:fact}
\begin{tabular}{llll} 
\toprule
\openup 4pt
$T$ & $d=\gcd(\cdot,\cdot)$ & cyclotomic factorisation of $d|T|$ & $M$\\ 
\midrule
$A_n$ & $(n+1,q-1)$ & ${\displaystyle q^{\binom{n+1}{2}}\,\Phi_1^n\prod_{i=2}^{n+1}\Phi_i^{\lfloor\frac{n+1}{i}\rfloor}}$ & $n+1$ \\[.5cm]
${}^2A_n$ & $(n+1,q+1)$ & 
$q^{\binom{n+1}{2}}\, \Phi_2^n \kern-4mm
\displaystyle\prod_{2 \neq i\equiv 2 (4)}
\kern-5mm
\Phi_i^{\lfloor\frac{2(n+1)}{\vphantom{l}i}\rfloor}
\displaystyle\prod_{i\not\equiv 2 (4)}\Phi_i^{\lfloor\frac{n+1}{\textup{lcm}(2,i)}\rfloor}$ 
& $\hspace{-5pt} \begin{array}{l} 2(n+1),\;  2\mid n \\ 2n, \hskip11.4mm 2\nmid n \end{array}$
\\[.5cm]
$B_n$, $C_n$ &$(2,q-1)$
&${\displaystyle q^{n^2}
\prod_{i=1}^{2n}\Phi_i^{\lfloor \frac{2n}{\textup{lcm}(2,i)}\rfloor}}$ & $2n$ \\[.5cm]
$D_n$&$(4,q^n-1)$
&$q^{n(n-1)}\!\!\!\!\!\displaystyle\prod_{i\,\nmid\, n \textup{ and } i\mid 2n}\!\!\!\!\!\Phi_i^{\frac{2n}{i}-1}\kern-2mm
\displaystyle\prod_{i\mid n\textup{ or } i\, \nmid\, 2n}\!\!\!\Phi_i^{\lfloor \frac{2n}{\textup{lcm}(2,i)}\rfloor}$  & $2(n-1)$\\[.5cm]
${}^2D_n$&$(4,q^n-1)$
&$q^{n(n-1)}\,\,\,\,\,\displaystyle\prod_{i\nmid n}\Phi_i^{\lfloor \frac{2n}{\textup{lcm}(2,i)}\rfloor}\,\displaystyle\prod_{i\,\mid\, n}\Phi_i^{\lfloor \frac{2n}{\textup{lcm}(2,i)}\rfloor-1}$
&$2n$\\ \bottomrule
\end{tabular}
\end{table}

We need the following lemma about cyclotomic polynomials.

\begin{lemma} \label{lemma:pDividesTwoPhis}
Let $i<j$ be integers, and suppose that $r$ is a prime dividing both $\Phi_i(q)$ and $\Phi_j(q)$ for some prime power $q$. 
Then $j/i = r^k$ for some positive integer $k$. 
In particular, $r$ divides $j$.
\end{lemma}

\begin{proof}
This follows immediately from \cite{Dresden}*{Theorem~2} or \cite{Filaesta}*{Lemma~2}.
\end{proof}

The next lemma imposes an upper bound on the number of distinct cyclotomic polynomial factors of $|T|$ that can be divisible by a given prime distinct from the characteristic. 

\begin{lemma} \label{lemma:exponent}
Let $T=T(q)$ be a finite simple group of Lie type defined over a field of order $q$, and let $d|T|=q^{e_0}\prod_{i=1}^M\Phi_i(q)^{e_i}$ be the cyclotomic factorisation of $T$, where $d$ is the number of diagonal outer automorphisms of $T$.
If $r$ is a prime dividing $|T|$ but not $q$, then $r$ divides at most $\lfloor\log_r(M/m)\rfloor+1$ of the factors $\Phi_1(q)^{e_1},\dots,\Phi_M(q)^{e_M}$, where $m$ is the order of $q$ modulo $r$.
\end{lemma}

\begin{proof}
Since $r$ divides $|T|$ but not $q$, it divides some factor $\Phi_i(q)^{e_i}$ of $d|T|$. Hence $e_i>0$.
Moreover, the minimal such $i$ is the order $m$ of $q$ modulo $r$. 
By Lemma~\ref{lemma:pDividesTwoPhis}, $r$ might also divide some or all of $\Phi_{mr}(q)^{e_{mr}},\ldots,\Phi_{mr^k}(q)^{e_{mr^k}}$, where $k$ is maximal such that $\Phi_{mr^k}(q)$ divides $|T|$, but $r$ cannot divide any of the other $\Phi_j(q)^{e_j}$. 
Since $mr^k\le M$, $r$ divides at most $1+k = 1+\lfloor\log_r(M/m)\rfloor$
of the factors $\Phi_j(q)^{e_j}$ of~$d|T|$. 
\end{proof}

The final lemma bounds the contribution of each cyclotomic polynomial factor to the order of a finite simple classical group.

\begin{lemma} \label{QboundClassical}
Let $T=T(q)$ be a finite simple classical group defined over a field of order $q$, as listed in \textup{Table~\ref{mainTable}} or \textup{Table~\ref{T:fact}}. 
Let $Q(T)$ denote the largest factor of the form $\Phi_i(q)^{e_i}$ dividing $d|T|$.
Then 
\[
  Q(T) \le |T|^{a/n}, \quad \text{where} \quad 
  a = \begin{cases}
  1 \quad& \text{if $T$ has type $\Aa_n$, $\Bb_n$ or $\Cc_n$,} \\
  2 & \text{if $T$ has type ${}^2\Aa_n$, $\Dd_n$ or ${}^2\Dd_n$}.
  \end{cases}
\]
\end{lemma}

\begin{proof}
Suppose first that $n=1$. 
Then $T=\Aa_1(q)$ (with $q\ge4$) and $d|T|=q(q^2-1)$, so $Q(T)=\Phi_2(q)$. 
The desired bound is $Q(T)\le|T|$, and this holds because 
\[
|T|\ge\frac{q(q^2-1)}{2}= \frac{1}{2}q(q-1)(q+1) \ge q+1=\Phi_2(q).
\]
Suppose from now on that $n\ge2$. 
We need the following inequality, which follows from \cite{NP}*{Lemma 3.5}:
\begin{equation}\label{omega}
1-q^{-1}-q^{-2}<\prod_{i=1}^\infty(1-q^{-i})\le 1-q^{-1}-q^{-2}+q^{-3} \quad \text{for all} \quad q\ge 2.
\end{equation}
We now divide the proof into four cases.
  
\noindent {\sc Case 1:} $T \cong \Aa_n(q)$ with $n\ge2$. 
Here $d=\gcd(n+1,q-1)<q$ (see Table~\ref{T:fact}) so \eqref{omega} yields
\begin{equation}\label{|T|boundAn}
|T| = \frac{q^{\frac{n(n+1)}{2}}}{d} \prod_{i=2}^{n+1} (q^i-1)
= \frac{q^{n(n+2)}}{d} \prod_{i=2}^{n+1} (1-q^{-i})
> \frac{1-q^{-1}-q^{-2}}{q} \cdot q^{n^2+2n}.
\end{equation}
Suppose that a cyclotomic polynomial $\Phi_i(q)$ divides $|T|$. 
According to Table~\ref{T:fact}, we have $1 \le i \le n+1$, $e_1 = n$ and $e_i = \lfloor (n+1)/i \rfloor \le n+1$ for $i \ge 2$. 
We now show that $\Phi_i(q)^{e_i}\le q^{n+1}/(q-1)$. 
This is true when $i=1$ because $(q-1)^{n+1}<q^{n+1}$; if $2\le i\le n+1$ then $\Phi_i(q)$ divides $(q^i-1)/(q-1)$, and so
\[
\Phi_i(q)^{e_i} \le \left( \frac{q^i-1}{q-1} \right)^{(n+1)/i} 
< \frac{(q^i)^{(n+1)/i}}{(q-1)^{(n+1)/i}} \le \frac{q^{n+1}}{q-1}.
\]
Therefore, $Q(T)\le q^{n+1}/(q-1)$, and it follows from~\eqref{|T|boundAn} that
\[ 
Q(T)^n \le \frac{q^{n(n+1)}}{(q-1)^n} = \frac{q^{n(n+2)}}{q^n(q-1)^n}
< \frac{q^{n(n+2)}}{4(q-1)}  \le (1-q^{-1} - q^{-2}) \frac{q^{n(n+2)}}{d} = |T|.
\]
Hence, $Q(T) \le |T|^{1/n}$ as claimed.

\noindent {\sc Case 2:} $T \cong {}^2\Aa_n(q)$. 
Here $d=\gcd(n+1,q+1)\le q+1$, so
\begin{equation} \label{|T|bound2An}
|T| = \frac{q^{n(n+2)}}{d} \prod_{i=2}^{n+1} (1-(-q)^{-i})
>\frac{q^{n(n+2)}}{q+1} (1+q^{-3})\prod_{j=1}^\infty (1-q^{-2j}).
\end{equation}
That is, for the estimate we have omitted the factors in the original product with odd $i>3$.
Using \eqref{omega}, the identity $(1+q^{-3})/(1+q^{-1})=1-q^{-1}+q^{-2}$, and the inequality $1-q^{-1}+q^{-2}\ge 11/16$, we obtain
\[
\frac{1+q^{-3}}{q+1}\prod_{j=1}^\infty (1-q^{-2j})
>\frac{(1+q^{-3})(1-q^{-2}-q^{-4})}{q(1+q^{-1})}
\ge\frac{11}{16}\left(\frac{1-q^{-1}+q^{-2}}{q}\right)>\frac1{2q}.
\]
Together with \eqref{|T|bound2An}, this yields 
\begin{equation} \label{|T|bound2An-2}
|T|>\frac12 q^{n(n+2)-1}.
\end{equation}
We now show that $Q(T)\le q^{2(n+1)}$. 
If a cyclotomic polynomial $\Phi_i(q)$ divides $d|T|$, then $1 \le i \le 2(n+1)$ by Table~\ref{T:fact}.
If $i=1$ then $e_1 = \lfloor (n+1)/2 \rfloor \le (n+1)/2$ (by Table~\ref{T:fact}), so $\Phi_1(q)^{e_1} \le (q-1)^{(n+1)/2}< q^{2(n+1)}$. 
Similarly, when $i=2$, $e_2 = n$ and $\Phi_2(q)^{e_2}< (q^2-1)^n\le q^{2(n+1)}$. 
Finally, if $i \ge 3$ then $e_i \le 2(n+1)/i$, so
\[
\Phi_i(q)^{e_i} < (q^i-1)^{2(n+1)/i}< (q^i)^{2(n+1)/i} = q^{2(n+1)}.
\]
Therefore, we have $Q(T)\le q^{2(n+1)}$, which together with \eqref{|T|bound2An-2} yields
\[
Q(T)^n<q^{2n^2+2n}\le\frac{q^{2n^2+4n-2}}{4}=\left(\frac{q^{2(n+1)-1}}{2}\right)^2
<|T|^2.
\]
Thus, $Q(T)<|T|^{n/2}$ as claimed.

\noindent {\sc Case 3:} $T \cong \Bb_n(q)$ or $\Cc_n(q)$. 
Here $d=\gcd(2,q-1)\le2$, which together with \eqref{omega} gives
\[
|T| 
= \frac{q^{2n^2+n}}{d} \prod_{i=1}^n (1-q^{-2i}) 
> \frac{1-q^{-2}-q^{-4}}{2} \cdot q^{2n^2+n}\ge \frac{11}{32}q^{2n^2+n}
>\frac{q^{2n^2+n}}{3}.
\]
Suppose that $\Phi_i(q)$ divides $d|T|$. 
Then $1 \le i \le 2n$ and $e_i = \lfloor 2n/\lcm(2,i) \rfloor \le 2n/i$ by Table~\ref{T:fact}.
Since $\Phi_i(q)\le q^i-1<q^i$, we have
\[
\Phi_i(q)^{e_i}<(q^i)^{2n/i}=q^{2n}.
\]
Hence, $Q(T)<q^{2n}$, and because $n\ge2$ we have
\[
Q(T)^n<q^{2n^2}<\frac{q^{2n^2+n}}{3}<|T|.
\]
Therefore, $Q(T) \le |T|^{1/n}$ as claimed.
  
\noindent {\sc Case 4:} $T \cong \Dd_n(q)$ or ${}^2\Dd_n(q)$. 
Here $d=\gcd(4,q \pm 1)\le 4$ and
\[
|T| = \frac{q^{n^2-n}(q^n \pm 1)}{d} \prod_{i=1}^{n-1}(q^{2i}-1)
> \frac{q^{2n^2-n}}{4} \prod_{i=1}^\infty(1-q^{-2i})
> \frac{(1-q^{-2}-q^{-4}) q^{2n^2-n}}{4}.
\]
Since $\prod_{i=1}^\infty(1-q^{-2i})>(1-q^{-2}-q^{-4})\ge 11/16$ by \eqref{omega}, we have $|T|>q^{2n^2-n-3}$.
If $\Phi_i(q)$ divides $d|T|$ then $i \le 2n$ and $e_i\le 2n/i$ by Table~\ref{T:fact}, so $\Phi_i(q)^{e_i}\le (q^i-1)^{2n/i}< q^{2n}$. 
Hence, $Q(T)\le q^{2n}$, 
and so
$
Q(T)^n\le q^{2n^2}<(q^{2n^2-n-3})^2<|T|^2.
$
Whence, $Q(T) \le |T|^{2/n}$ as claimed. 
This completes the proof.
\end{proof}

\section{Proof of Theorem~\ref{mainThm}} \label{S3}

\begin{table}[!t]
\caption{Values of $Q(T)$ for the exceptional Lie type groups $T$, and constants $K$, $d_0$ and $q_0$ such that $Q(T)^K\le \frac{d}{d_0}|T|$ for all $q\ge q_0$.} \label{T:Q}
\begin{tabular}{lcccccccccc}
\toprule
$T$ & ${}^2\Bb_2$ & ${}^3\Dd_4$ & $\Ee_6$ & ${}^2\Ee_6$ & $\Ee_7$ & $\Ee_8$ & $\Ff_4$ & ${}^2\Ff_4'$ & $\Gg_2$ & ${}^2\Gg_2$ \\
\midrule
$Q(T)$ & $\Phi_4$ &$\Phi_3^2$&$\Phi_3^3$&$\Phi_2^6$&$\Phi_2^7$&$\Phi_2^8$&$\Phi_2^4$&$\Phi_4^2$&$\Phi_2^2$&$\Phi_6$ \\
$K$&2&6&12&12&18&29&12&6&6&$7/2$\\
$d_0$&1&1&3&3&2&1&1&1&1&1\\
$q_0$&8&4&5&7&9&7&7&8&4&27\\
\bottomrule
\end{tabular}
\end{table}

\begin{proof}[Proof of Theorem~\ref{mainThm}]

As in Section~\ref{S2}, write $d|T|=q^{e_0}\prod_{i=1}^M\Phi_i(q)^{e_i}$ and let $Q(T)$ denote the maximum of $\Phi_1(q)^{e_1},\dots,\Phi_M(q)^{e_M}$. 
Let $r$ be a prime dividing $|T|$ but not $q$, and let $R$ be a Sylow $r$-subgroup of $T$.
It follows from Lemma~\ref{lemma:exponent} that 
\begin{equation} \label{final}
|R|\le Q(T)^{1+\lfloor\log_r(M)\rfloor}.
\end{equation} 
We now seek `large' constants $K$ satisfying $Q(T)^K\le|T|$, in order to deduce the claimed bounds of the form $|R| \le |T|^{(\lfloor\log_r(M)\rfloor+1)/K}$.

Suppose first that $T$ is a classical group. 
Then the values of $M$ may be deduced from the cyclotomic factorisations listed in Table~\ref{T:fact}, and are as listed in both Table~\ref{mainTable} and Table~\ref{T:fact}. 
The values of $K$ appearing in Table~\ref{mainTable} are obtained from Lemma~\ref{QboundClassical}.

Now suppose that $T$ is an exceptional Lie type group. 
Then $d|T|=q^{e_0}\prod_{i=1}^{30}\Phi_i(q)^{e_i}$, where the values of $e_1,\dots, e_{30}$ are listed in \cite{KLST}*{Table~C.2}; in particular, we obtain the values of $M$ appearing in Table~\ref{mainTable}.
The values of $e_0$ appear in the last row of \cite{KLST}*{Table~C.3}, and in \cite{Atlas}*{Table~6}.
By inspecting these factorisations, one finds the value of $i$ such that $Q(T) = \Phi_i(q)^{e_i}$ (and that this value is independent of $q$). 
Table~\ref{T:Q} lists constants $K$, $d_0$ and $q_0$ such that 
\[
\Phi_i(q)^{e_iK} \le \frac{d}{d_0}|T| \quad \text{for all} \quad q\ge q_0.
\] 
Note that $d=d_0=1$ in all cases, except when $T=\Ee_6(q)$, ${}^2\Ee_6(q)$ or $\Ee_7(q)$, where $(d,d_0)=(\gcd(3,q-1),3)$, $(\gcd(3,q+1),3)$ or $(\gcd(2,q-1),2)$, respectively. 
In particular, $d\le d_0$ in all cases, and so
\[
Q(T)^K \le |T| \quad \text{for all} \quad q\ge q_0.
\]
The constants $K$ agree with those in Table~\ref{mainTable}, and so by combining the above bound with \eqref{final} we obtain the claimed bounds $|R| \le |T|^{(\lfloor\log_r(M)\rfloor+1)/K}$ for $q \ge q_0$. 
It remains to consider the cases where $q<q_0$. 
In these cases, we check manually, for each prime $r$ dividing $|T|$ but not $q$, whether the bound $|R| \le |T|^{(\lfloor\log_r(M)\rfloor+1)/K}$ (with $K$ and $M$ as in Table~\ref{mainTable}) holds. 
The exceptions, which were checked both manually and using the {\sc Magma} code available at the first author's website\footnote{\url{http://school.maths.uwa.edu.au/~glasby/ExceptionsCheck.mag}}, are recorded in Table~\ref{mainTable}.
\end{proof}

\begin{remark}
  One can slightly improve the values of $K$ listed
  in Table~\ref{mainTable} in some cases. For example, if
  $T \cong {}^2\Bb_2(q)$ then $K$ can be increased to
  $\log(|{}^2\Bb_2(8)|)/ \log(\Phi_2(8))\approx 2.46$. As such improvements
  seem tedious to achieve and do not change the form of our generic bound $|R| \le |T|^{O(\log_r(\ell)/\ell)}$, where $\ell$ is the
  Lie rank of $T$, we chose not to pursue them here. 
\end{remark}


\section*{Acknowledgements}
All three authors acknowledge support from the Australian Research
Council (ARC) grant DP140100416, and SPG also acknowledges support from DP160102323. 
SPG and TP are grateful to RWTH Aachen University for financial support and hospitality during their respective visits in 2017, when the research leading to this paper was undertaken. 

\end{document}